\documentclass[preprint,12pt]{elsarticle}
\pdfoutput=1
\journal{Topology and its Applications}

\usepackage[utf8]{inputenc}
\usepackage{amsfonts,amssymb, amsmath, amsthm}
\usepackage{mathrsfs}
\usepackage{mathtools} 
\usepackage{xcolor}
\usepackage{enumitem}

\usepackage{tikz}
\usepackage{tikz-cd}

\usepackage[colorlinks]{hyperref}

\hypersetup{
    pdfauthor={Tom{\'a}{\v s} Jakl},
    pdftitle={Canonical extensions of locally compact frames},
    pdfkeywords={canonical extensions, frames, pointfree topology, duality theory},
    linkbordercolor = {0 0 0},
    citebordercolor = {0 0 0},
    urlbordercolor  = {0 0 0},
}

\definecolor{darkpastelgreen}{rgb}{0.01, 0.75, 0.24}
\definecolor{forestgreen}{rgb}{0.13, 0.55, 0.13}

\AtBeginDocument{%
\hypersetup{
    urlcolor  = forestgreen,
    linkcolor = forestgreen,
    citecolor = forestgreen,
    }
}

\newcommand\qq[1]{\quad #1 \quad}
\newcommand\ee[1]{\enspace #1 \enspace}
\newcommand\ete[1]{\enspace\text{#1}\enspace}
\newcommand\qtq[1]{\quad\text{#1}\quad}
\newcommand\qqtqq[1]{\quad\quad\text{#1}\quad\quad}
\newcommand\p[1]{\ensuremath{\mathscr{ #1 }}}

\newcommand{\inv}{^{-1}}
\newcommand\op{^\text{op}}
\newcommand\onto{\twoheadrightarrow}
\newcommand\into{\hookrightarrow}
\newcommand{\ttimes}{\mathclose{\times}}
\newcommand\sue{\subseteq}

\newcommand\upset  {\ensuremath{\mathord{\uparrow}\mkern1mu}}
\newcommand\downset{\ensuremath{\mathord{\downarrow}\mkern1mu}}
\newcommand\id{\ensuremath{\mathsf{id}}}
\newcommand\Up{\ensuremath{\mathsf{Up}}}
\newcommand\Idl{\ensuremath{\mathsf{Idl}}}
\newcommand\RIdl{\ensuremath{\mathsf{RIdl}}}
\newcommand\Filt{\ensuremath{\mathsf{Filt}}}

\newcommand{\spec}{\ensuremath{\mathsf{spec}}}

\def\mee{\wedge}
\def\bigmee{\bigwedge}
\DeclareMathOperator*{\dirvee}{\bigvee\dirup}
\DeclareMathOperator*{\dircup}{\bigcup\dirup}
\newcommand\dirup{\!{}^{\upset}}
\newcommand\dirdown{\!{}^{\downset}}

\newcommand\KS{\ensuremath{\mathsf{KS}}}
\newcommand\SO{^{\wedge}}
\newcommand\iSO{^{\dagger}}
\newcommand\can{^\delta}

\newcommand{\two}{\ensuremath{\mathbf{2}}}
\newcommand\Oo{\ensuremath{\mathcal O}}
\newcommand\pt{\ensuremath{\mathsf{pt}}}
\newcommand\ptC{\ensuremath{\mathsf{pt}_{\bigmee}}}
\newcommand\Sl{\ensuremath{\mathcal{S}}}
\newcommand\Sc{\ensuremath{\mathcal{S}_c}}
\newcommand\Ps{\raisebox{.17\baselineskip}{\Large\ensuremath{\wp}}}

\newcommand\Top{\ensuremath{\textbf{Top}}}
\newcommand\Frm{\ensuremath{\textbf{Frm}}}

\newcommand\emb{\ensuremath{e\colon L\to L\can}}
\newcommand\embI{\ensuremath{e\colon L\hookrightarrow L\can}}
\newcommand\ARG{\mkern1mu\cdot\mkern1mu}
\newcommand\wbelow{\ll}

\newcommand\soI{\ensuremath{\kappa}}
\newcommand\soII{\ensuremath{\lambda}}
\newcommand\soIII{\ensuremath{\chi}}

\newenvironment{axioms}
    {\begin{enumerate}[labelindent=4em, leftmargin=*, itemsep=0.8em, topsep=1em]}
    {\end{enumerate}}

\newenvironment{diagram}[1][]{\begin{center}\begin{tikzcd}[#1]}{\end{tikzcd}\end{center}}

\tikzcdset{
    cramped/.code={\tikzcdset{
        every matrix/.append style={inner sep=+-0.3em},
        every cell/.append style={inner sep=+0.3em}}},
    sep/.code={\tikzcdset{row sep={#1},column sep={#1}}}
}

\newcommand{\cdmatrix}[2][]{%
    \begin{tikzcd}[cramped, sep=2em, ampersand replacement=\&,#1]%
    #2 %
    \end{tikzcd}%
}

\theoremstyle{plain}
\newtheorem{theorem}{Theorem}[section]
\newtheorem{corollary}[theorem]{Corollary}
\newtheorem{lemma}[theorem]{Lemma}
\newtheorem{proposition}[theorem]{Proposition}
\newtheorem{observation}[theorem]{Observation}
\newtheorem{fact}[theorem]{Fact}

\theoremstyle{definition}
\newtheorem{example}[theorem]{Example}
\newtheorem{definition}[theorem]{Definition}
\newtheorem{remark}[theorem]{Remark}

\begin{document}
\begin{frontmatter}
\title{Canonical extensions of locally compact frames}
\author{Tom\'a\v s Jakl\corref{grant}}
\cortext[grant]{
    The research discussed has received funding from the European
        Research Council (ERC) under the European Union's Horizon 2020
        research and innovation programme (grant agreement No.670624).}
\address{Laboratoire J. A. Dieudonn\'e, CNRS and Universit\'e C\^ote d'Azur,\\
    06108 Nice Cedex 02, France}
\ead{tomas.jakl@unice.fr}
\date{December 1, 2018}

\begin{abstract}
    Canonical extension of finitary ordered structures such as lattices, posets, proximity lattices, etc., is a certain completion which entirely describes the topological dual of the ordered structure and it does so in a purely algebraic and choice-free way. We adapt the general algebraic technique that constructs them to the theory of frames.

    As a result, we show that every locally compact frame embeds into a completely distributive lattice by a construction which generalises, among others, the canonical extensions for distributive lattices and proximity lattices. This construction also provides a new description of a construction by Marcel Ern\'e. Moreover, canonical extensions of frames enable us to frame-theoretically represent monotone maps with respect to the specialisation order.
\end{abstract}


\end{frontmatter}

\section{Introduction}

Initially canonical extensions were defined by J\'onsson and Tarski~\cite{jonssontarski1951boolean, jonssontarski1952boolean} in order to understand topological duals of Boolean algebras with additional operations. The theory of canonical extensions was later extended to distributive lattices, posets, proximity lattices, coherent categories and more~\cite{gehrkejonsson2004bounded,gehrkejonsson1994bounded,gehrkejansanapalmigiano2013delta,gehrkeharding2001bounded,gehrkepriestley2008canonical,gool2012duality,coumans2012generalising}.

Intuitively, canonical extension of a distributive lattice algebraically represents the lattice embedding of the compact open subsets of its Stone dual into the complete lattice of saturated sets, that is, upsets with respect to the specialisation order. Because the construction is purely algebraic, canonical extension of a distributive lattice can be thought of as a point-free and choice-free description of the frame of all saturated subsets of its dual space.

There is nothing specific about distributive lattices and their Stone duals. We can ask whether there is an analogous construction for frames and topological spaces. In particular, the frame of open sets of a topological space always embeds into the frame of saturated subsets. The problem therefore is to give a choice-free and purely frame-theoretic or algebraic construction of the frame of saturated subsets.

We solve this question for the class of locally compact frames. This is a wide class of frames which includes all compact regular frames as well as coherent frames. In fact our technique works for an even more general class of frames, that is, for frames with enough compact fitted sublocales.
 We show that every locally compact frame canonically embeds into a complete lattice, which it meet-generates, and which is isomorphic to the frame of saturated subsets of the dual space of the frame we started with. From this it follows that the canonical extension of the frame is a completely distributive lattice, although this fact is not choice-free. Moreover, we also show how to lift perfect homomorphisms between locally compact frames to complete lattice homomorphisms between their canonical extensions.

Our construction generalises the distributive case of Sam van Gool~\cite{gool2012duality} which solves this question for proximity lattices and, therefore, also indirectly for stably compact frames. Moreover, it turns out that our construction is exactly the same as the one considered by Marcel Ern\'e in~\cite{erne2007choiceless} in which he constructs a certain frame extension for every preframe. We also show how one can use canonical extensions to represent monotone maps with respect to specialisation order, in a point-free way.

\section{Preliminaries}

For a map $f\colon X\to Y$ between sets $X$ and $Y$, it will be important throughout the text to notationally distinguish its application to one element, from its forward image map $f[\ARG]\colon \Ps(X) \to \Ps(Y)$, and the preimage map $f\inv[\ARG]\colon \Ps(Y) \to \Ps(X)$, that is, from the maps
\[ f[M] = \{ f(x) ~|~ x\in M\}, \ete{and} f\inv[N] = \{ x \in X ~|~ f(x) \in N\}. \]

In the following we summarise the main facts on order theory and point-free topology which we will rely on.
Proofs of most of the results summarised in this section can be found in the standard reference book~\cite{picadopultr2012book}.

\subsection{Topological spaces}


Every topological space $(X, \tau)$ carries the so-called \emph{specialisation order} on the set of its points:
\[ \forall x,y\in X.\quad x\leq_\tau y \qtq{iff} \forall U \in \tau.\ x\in U \implies y\in U. \]
\emph{Saturated} subsets of a space are those subsets which are upward closed with respect to the specialisation order. We denote the set of all saturated subsets/upsets of a space $(X, \tau)$ by $\Up(X, \leq_\tau)$, or sometimes just $\Up(X)$.  Note that a subset of a space $M\sue X$ is saturated if and only if $M = \bigcap \{ U\in \tau ~|~ M \sue U \}$ or equivalently if there is a collection $\p U\sue \tau$ such that $M = \bigcap \p U$ (Proposition 4.2.9 in~\cite{goubault2013nonhausdorf}). Note also that all open subsets of a space are saturated and so $\tau \sue \Up(X)$.



\subsection{Galois adjunctions}
Order preserving function $f\colon A \to B$ between complete lattices preserves arbitrary joins if and only if there is a monotone function $g\colon B\to A$ (in fact, meet  preserving) such that
\[ f(a) \leq b \qtq{iff} a \leq g(b), \]
or equivalently
\[ f(g(b)) \leq b \qtq{iff} a \leq g(f(a)). \]
We call $f$ the \emph{left (Galois) adjoint} and $g$ the \emph{right (Galois) adjoint}. It follows that $f = f g f$ and $g = g f g$ and so by setting $c = g f$ we obtain a \emph{closure operator} $c\colon A\to A$ on $A$ which satisfies the following conditions, for all $a, a'\in A$:
\[ a \leq a' \implies c(a) \leq c(a'),\quad a \leq c(a),\qtq{and} c(c(a)) \leq c(a). \]
The fixpoints of $c\colon A\to A$ form a complete lattice, in the order inherited from $A$.

Note that if we start from two \emph{antitone} maps $f\colon A\rightleftarrows B \colon g$ such that $a \leq g(b)$ iff $b\leq f(a)$ then we speak of an \emph{antitone Galois adjunction}. In this case both maps $f$ and $g$ transform arbitrary joins into meets. In fact, antitone Galois adjunctions are historically the more standard setting.

\subsection{Frames}

A complete lattice $L$ is a \emph{frame} if, for every $A\sue L$ and $b\in L$,
\[ (\bigvee A)\mee b = \bigvee_{a\in A} (a\mee b). \]
Similarly, a \emph{preframe} $P$ is a poset which has finite meets, directed joins and, for every \emph{directed} $D\sue P$ and $b\in P$,
\[ (\dirvee D)\mee b = \dirvee_{d\in D} (d\mee b), \]
where $\dirvee D$ denotes the supremum of the directed set $D$. Homomorphisms of frames (resp.\ preframes) are the maps which preserve arbitrary joins  (resp.\ directed joins) and finite meets. We say that a poset is a \emph{coframe} if it is a frame in the opposite order and, similarly, it is a \emph{co-preframe} if it is a preframe in the opposite order.

Every topological space $(X, \tau)$ gives rise to a frame $\Oo(X, \tau) = \tau$ and, vice versa, every frame $L$ gives rise to a space $\pt(L)$ whose points are the frame homomorphisms $p\colon L\to \two$ into the two-element frame \two. The topology of $\pt(L)$ consists of the sets of the form $\{ p\in \pt(L) ~|~ p(a) = 1\}$, for some $a\in L$. Moreover, the mappings $X\mapsto \Oo(X)$ and $L\mapsto \pt(L)$ extend to a pair of dually adjoint functors, with the actions on morphisms defined as
\begin{align*}
    f\colon X\to Y &\ee\longmapsto \Oo(f)\colon \Oo(Y) \to \Oo(X),\ V\mapsto f\inv[V],\ete{and}\\
    h\colon L\to M &\ee\longmapsto \pt(h)\colon \pt(M) \to \pt(L),\ p\mapsto p\circ h.
\end{align*}

In case when a frame is isomorphic to $\Oo(X)$, for some space $X$, we say that the frame is \emph{spatial} or that it has \emph{enough points}. Conversely, spaces which are homeomorphic to $\pt(L)$, for some frame $L$, are called \emph{sober}.

\subsection{Scott-open filters and locally compact frames}
A \emph{filter} $F\sue A$ of a frame or distributive lattice\footnote{We assume that all lattices are bounded.} $A$ is an upset closed under finite meets. In particular, $F$ is non-empty as it contains the empty meet, i.e.\ the top element. Similarly, a subset $I\sue A$ is an \emph{ideal} if it is a filter in $A\op$, that is, in $A$ ordered by the opposite order. The frame of ideals ordered by set inclusion is denoted by $\Idl(A)$. Similarly, the \emph{co}frame of filters ordered by the reverse inclusion is denoted by $\Filt(A)$. Note that filtered meets, binary joins and binary meets of filters are computed as follows:
\[ \bigmee_i\dirdown\ F_i = \dircup F_i,\quad F \vee G = F\cap G,\quad F\mee G = \{ f\mee g ~|~ f\in F,\ g\in G\}. \]

A filter $F$ of a frame $L$ is called \emph{Scott-open} if, whenever $\dirvee D\in F$ for some directed $D\sue L$, then already $d\in F$ for some $d\in D$. Note that the poset of Scott-open filters on a frame $L$ is a sub-co-preframe of $\Filt(L)$, which we will denote by $L\SO$. (In our notation, we will always denote Scott-open filters by $\soI$, $\soII$, or $\soIII$.) Moreover, Scott-open filters are in a bijective correspondence with \emph{pre}frame homomorphisms $L\to \two$. Consequently, a preimage of a Scott-open filter by a preframe homomorphism is again a Scott-open filter.

We prefer to order filters by reverse inclusion because then we can interpret filters as a point-free representation of saturated subsets. This is justified because such order makes the following two maps monotone (for a sober space $X$)
\[ M \in \Up(X) \mapsto \{ U\in \Oo(X) ~|~ M \sue U\} \qtq{and} F\in \Filt(\Oo(X)) \mapsto \bigcap F.\]
This representation of saturated subsets as filters is not unique in general. However, if we restrict to Scott-open filters then we obtain a unique representation of compact saturated subsets. This is the celebrated \emph{Hofmann--Mislove Theorem}~\cite{hofmannmislove1981local, keimelpaseka1994direct}.

A frame $L$ is called \emph{locally compact} (or \emph{continuous}), if $a = \dirvee \{ c ~|~ c\wbelow a\}$ for every $a\in L$, where
\[ c\wbelow a \qtq{iff} a \leq \dirvee D \implies c\leq d \text{ for some } d\in D.\]
For locally compact frames this relation has also another characterisation (e.g.\ see Lemma~VII.6.3.2 in~\cite{picadopultr2012book} or Proposition I-3.3 in~\cite{gierzetal2003continuous}):
\begin{align*}
    c\wbelow a \qtq{iff} \exists \soI\in L\SO \text{ s.t.\ } a\in \soI \sue \upset c
\end{align*}
The dual adjunction between the functors $\Oo$ and $\pt$ restricts to a duality of categories, with locally compact frames on one side and \emph{locally compact sober} spaces on the other side. This duality is known as the Hofmann--Lawson duality.

\subsection{Sublocales}
Next we introduce the frame-theoretic analogue of subspaces. A subset $S$ of a frame $L$ is a \emph{sublocale} if $S$ is the set of fixpoints of some \emph{nucleus} $\nu\colon L\to L$, where nucleus is a closure operator which preserves binary meets.

Every sublocale $S$, in the order derived from $L$, is a frame and the nucleus $\nu$ restricts to an \emph{onto} frame homomorphism $L\onto S$. In fact, every onto frame homomorphism gives rise to a nucleus, and vice versa.

Finally, the lattice of sublocales $\Sl(L)$, ordered by set inclusion, forms a coframe with meets computed as intersections, and embeds the original frame $L$ via the map
\[ a\in L \ee\mapsto U_a = \{ x ~|~ x = x \to a \}.\footnotemark \]
\footnotetext{Here $\to$ is the Heyting implication of the frame.}
Sublocales of the form $U_a$ are called \emph{open sublocales} and their complements in $\Sl(L)$, which always exist, are called \emph{closed sublocales}. For a more in-depth discussion about sublocales, nuclei and frame quotients we refer the reader to \cite{picadopultr2012book}.

\section{Abstract description}

\subsection{Canonical extensions for distributive lattices}\label{s:abstract-dlat}
To motivate our definition of canonical extension for frames we first recall the construction for distributive lattices (for details see e.g.~\cite{dunngehrkealessandra2005canonical,gehrke2014canonical,gehrkeharding2001bounded,gehrkejonsson1994bounded}). Given a distributive lattice $A$, its canonical extension $e\colon A\into A\can$ is a lattice embedding of $A$ into a complete lattice $A\can$ uniquely described by the following two conditions

\begin{axioms}
    \item[(Dense)] For all $u,v \in A\can$, $u\leq v$ if, for every filter $F\sue A$ and ideal $I\sue A$, $\bigmee e[F] \leq u$ and $v\leq \bigvee e[I]$ implies $\bigmee e[F] \leq \bigvee e[I]$.
    \item[(Compact)] If $\bigmee e[F] \leq \bigvee e[I]$ for some filter $F\sue A$ and ideal $I\sue A$, then $F \cap I \not= \emptyset$.
\end{axioms}

Density simply means that every element of the canonical extension $A\can$ is both a join of meets and a meet of joins of elements of $A$.

Note that up until recently it was preferred to write $A^\sigma$ to denote the canonical extension of $A$. However, this notation conflicts the notation for the two different extensions, $\sigma$-- and $\pi$-extension, of monotone maps -- the first defined as a join of meets and the other as a meet of joins (and there is no reason why one should have any preference over the other).

\begin{remark}\label{r:dlat-expl}
    The intuition behind the two axioms is that $e\colon A\into A\can$ represents the embedding of compact opens of $X$ into the complete lattice of saturated subsets of $X$, where $X$ is the Stone dual (spectral) space of $A$. Under this interpretation, ideals of $A$ correspond to the open sets via the mapping $I \mapsto \bigvee e[I]$ and, similarly, filters correspond to the compact saturated subsets via $F \mapsto \bigmee e[F]$.

    Density then simply means that every saturated subset $M$ of $X$ is uniquely determined by compact saturated sets $K$ such that $K \sue M$ and also by open sets $U$ such that $M \sue U$. Further, compactness means that $K \sue U$ implies that there is a compact open $O$ such that $K\sue O \sue U$.

    Note that, by assuming the axiom of choice, we can identify the elements of $A\can$ that correspond to the principal upsets $\upset x$ in the specialisation order of $X$ as those elements which are completely join prime. We see that $A\can$ represents the points of $X$ as well as its upsets in one structure and so it effectively contains all information about the space $X$.~\cite{gehrke2014canonical}
\end{remark}

Existence and uniqueness of the canonical extension $A\into A\can$ follow from a general theory of polarities as described, for example, in~\cite{gehrke2006generalized}. Let us recall some of the basics of the theory that goes back to Birkhoff~\cite{birkhoff1979lattice}. A \emph{polarity} is a triple $(X, Y, Z)$ where $X$ and $Y$ are sets and $Z \sue X\ttimes Y$ is a relation. This data induces the following pair of antitone maps between $\Ps(X)$ and $\Ps(Y)$, in the subset order:
\begin{align*}
    p\colon \Ps(X) \to \Ps(Y), \quad M \mapsto \{ y\in Y ~|~ \forall x\in M.\ x Z y\} \\
    q\colon \Ps(Y) \to \Ps(X), \quad N \mapsto \{ x\in X ~|~ \forall y\in N.\ x Z y\}
\end{align*}
Since $N \sue p(M)$ iff $M \sue q(N)$, the maps $p$ and $q$ constitute an antitone Galois adjunction and $\phi = q\circ p$ is a closure operator on $\Ps(X)$. Set $\p G(X,Y,Z)\sue \Ps(X)$ to be the set of Galois closed sets, i.e.\ sets $\{ M\sue X ~|~ \phi(M) = M \}$. There are also two maps
\[ f\colon X\to \p G(X,Y,Z),\ x\mapsto \phi(\{x\}),
   \ete{and}
   g\colon Y\to \p G(X,Y,Z),\ y\mapsto q(\{y\}). \]

\begin{fact}[{e.g., Section 2 in \cite{gehrke2006generalized}}]\label{f:polarities}
    Let $(X, Y, Z)$ be a polarity. Then the complete lattice $C = \p G(X, Y, Z)$ has the following properties.
    \begin{enumerate}
        \item For any $u \in C$,
        \[
            u = \bigvee \{ f(x) ~|~ x\in X,\ f(x)\leq u\}
            \ete{and}
            u = \bigmee \{ g(y) ~|~ y\in Y,\ u \leq g(y)\}.
        \]
        \item For any $x\in X$ and $y\in Y$, $f(x) \leq g(y)$ iff $x Z y$.
        \item For any $f'\colon X \to C'$ and $g'\colon Y \to C'$ also satisfying the conditions (1) and (2) there is a unique complete lattice isomorphism $\iota\colon C'\to C$ such that $\iota \circ f' = f$ and $\iota \circ g' = g$.
    \end{enumerate}
\end{fact}

Consequently, for a distributive lattice $A$, setting $X = \Filt(A)$, $Y = \Idl(A)$, and $Z = \{ (F, I) ~|~ F \cap I \not= \emptyset\}$ gives us that, for the complete lattice $C = \p G(X, Y, Z)$, the composite
\[ e\colon A\xrightarrow{~a~\mapsto~\downset a~} \Idl(A) \xrightarrow{\qq g} C \]
is the canonical extension of $A$, e.g.~\cite{gehrkeharding2001bounded}.

\subsection{The same construction for spaces}
\label{s:construction-spaces}
The construction mentioned above can be rewritten purely in terms of the Stone dual $(X, \tau)$ of the distributive lattice $A$. As we mentioned in Remark~\ref{r:dlat-expl}, the frame of ideals is isomorphic to $\tau$ and the coframe of filters is isomorphic to the poset of compact saturated subsets $\KS(X)$ of $X$, ordered by set inclusion. Moreover, the relation $F Z I$ expresses the fact that $K_F \sue U_I$ where $K_F$ is the compact saturated subset corresponding to $F$ and $U_I$ is the open set corresponding to $I$.

This intuition allows us to obtain the frame ${\Up(X, \leq_\tau)}$ entirely from this data. Furthermore, this construction works for all topological spaces, not just those arising as spectra of distributive lattices.

\begin{proposition}\label{p:space-case}
    Let $(X, \tau)$ be a topological space. Then
    $\p G(\KS(X), \tau, Z)$ is isomorphic to ${\Up(X, \leq_\tau)}$
    where $K Z U$ iff $K \sue U$.
\end{proposition}
\begin{proof}
    By Fact~\ref{f:polarities} it is enough to check that the embeddings $f\colon \KS(X) \into \Up(X, \leq_\tau)$ and $g\colon \tau \into \Up(X, \leq_\tau)$ satisfy the conditions (1) and (2).

    The first condition is immediate as, for a saturated subset $M$ of $X$, $M = \bigcup \{ \upset x ~|~ x\in M\}$ and $M = \bigcap \{ U\in \tau ~|~ M \sue U \}$. The second condition follows immediately from the definition of $Z$.
\end{proof}

We see that the proof of Proposition~\ref{p:space-case} relies on the fact that the principal upsets in the specialisation order $\upset x$ are compact and saturated. In the following we will try to mimic this construction in a purely frame-theoretic setting.

\subsection{Canonical extension for frames}\label{s:abstract-frames}
By the Hofmann--Mislove Theorem, compact saturated subsets of a sober space are in a bijective correspondence with Scott-open filters of its lattice of opens. Therefore, we expect that the best candidate for the canonical extension of a frame $L$, by example of Proposition~\ref{p:space-case}, is the lattice $\p G(L\SO, L, Z)$ where, for $\soI\in L\SO$ and $a\in L$, $\soI Z a$ iff $a \in \soI$. The expression for $Z$ represents the statement that the compact saturated set corresponding to \soI{} is a subset of the open set corresponding to $a$.

Before we understand what is the structure of the lattice we obtained this way, let us first describe the canonical extension $L\can$ of $L$ based on the expected mutual relationship between Scott-open filters $L\SO$ and (abstract) open sets $L$ when embedded into $L\can$. Since it is expected that the compact saturated subset corresponding to a Scott-open filter is obtained as the ``intersection'' of all the opens it contains, it is natural to consider the following definition.

\begin{definition}\label{d:can-ext}
    Let $L$ be a frame. A monotone mapping \emb{} into a complete lattice $L\can$ is a \emph{canonical extension} of $L$ if
\begin{axioms}
    \item[(Dense)] For all $u,v \in L\can$, $u\leq v$ if, for all $\soI\in L\SO$ and $a\in L$, satisfying $\bigmee e[\soI] \leq u$ and $v\leq e(a)$ implies $\bigmee e[\soI] \leq e(a)$.

    \item[(Compact)] If $\bigmee e[\soI] \leq e(a)$ for some $\soI\in L\SO$ and $a\in L$, then $a\in \soI$.
\end{axioms}
\end{definition}
Note that $L\can$ is not required to be distributive, nor even a frame.

\begin{example}\label{e:can-ext-space}
    Let us check that the definition holds for our motivating example. We check that, for a sober space $(X, \tau)$, the embedding $\tau \into \Up(X, \leq_\tau)$ is a canonical extension of $\tau$.

    Compactness follows immediately from the Hofmann--Mislove Theorem because, for a Scott-open filter $\soI\sue \tau$ and an open $U\in \tau$,
    $\bigcap \soI \sue U$ iff $U\in \soI.$
    For density, let $M \not\sue N$ for some upsets $M$ and $N$. Since $N$ is equal to $\bigcap \{ U\in \tau ~|~ N \sue U\}$, there must be a $U\in \tau$ such that $M\not\sue U$ and $N\sue U$. Pick an $x\in M\setminus U$. Then $\upset x \not\sue U$ but $\upset x \sue M$. For the Scott-open filter $\soI =  \{ V\in \tau ~|~ \upset x \sue V\}$ corresponding to $\upset x$ we see that $\bigcap \soI \not\sue U$.
\end{example}

We postpone the proof of existence and uniqueness of canonical extensions for the moment and focus on the properties that follow immediately from the definition.

Since the expression $\bigmee e[\soI]$ appears repeatedly in our theory, we define a shorthand. For a map \emb, set
\[ e\SO\colon L\SO \to L\can,\quad \soI \mapsto \bigmee e[\soI] = \bigmee \{ e(a) ~|~ a\in \soI \}. \]

\begin{proposition}\label{p:basic-properties}
    Let $L$ be a frame and \emb{} be its canonical extension. Then
    \begin{enumerate}
        \item \emb{} is a preframe homomorphism preserving 0.
        \item $e\SO\colon L\SO\to L\can$ is an injective co-preframe homomorphism \\
              (recall that $L\SO$ is taken to be ordered by reverse inclusion).
        \item For every $u\in L\can$,
           \begin{align*}
               u &= \bigmee \{ e(a) ~|~ a\in L,\ u \leq e(a) \}
               \ete{and}\\
               u &= \bigvee \{ e\SO(\soI) ~|~ \soI\in L\SO,\ e\SO(\soI) \leq u \}.
           \end{align*}
        \item If \emb{} is injective, then it is a frame homomorphism.
        \item If $L$ is locally compact, then $e$ is an injective frame homomorphism.
    \end{enumerate}
\end{proposition}
\begin{proof}
    (1) First we check that $e$ preserves 0 and 1. For the former, by density of $e$, it is enough to check that whenever $\bigmee e[\soI] \leq e(0)$ and $0 \leq e(a)$ then also $\bigmee e[\soI] \leq e(a)$. From compactness applied to the first inequality we get that $0 \in \soI$ and, from monotonicity, $e(0) \leq e(a)$. Hence, $\bigmee e[\soI] \leq e(0) \leq e(a)$. We check that $1 \leq e(1)$ by a similar argument. Let $\soI\in L\SO$ and $a\in L$ be such that $\bigmee e[\soI] \leq 1$ and $e(1) \leq e(a)$. Since $1 \in \soI$ and $e$ is monotone, $\bigmee e[\soI] \leq e(1) \leq e(a)$.

    To check that $e$ preserves finite meets and directed joins we only need to check one inequality because $e$ is monotone. Let $a, b\in L$. Assume that $\bigmee e[\soI] \leq e(a)\mee e(b)$ and $e(a\mee b) \leq e(c)$ for some $\soI\in L\SO$ and $c\in L$. By compactness $a, b\in \soI$. Because \soI{} is a filter, $a\mee b\in \soI$ and so $\bigmee e[\soI] \leq e(c)$.

    Lastly, to check that $e(\dirvee D) \leq \dirvee e[D]$ for some directed $D\sue L$, assume that $\bigmee e[\soI] \leq e(\dirvee D)$ and that $\dirvee e[D] \leq e(a)$ for some $\soI\in L\SO$ and $a\in L$. Compactness and the fact that $\soI$ is Scott-open give that $d\in \soI$ for some $d\in D$ and, consequently, $\bigmee e[\soI] \leq e(d) \leq \dirvee e[D] \leq e(a)$.

    (2) For injectivity assume that $e\SO(\soI) = e\SO(\soII)$. Then, for an $a\in \soII$, $e\SO(\soI) = e\SO(\soII) \leq e(a)$ and so, by compactness, $a\in \soI$. Hence, $\soII\sue \soI$. The reverse inclusion is analogous.

    $e\SO$ preserves the bottom as $e\SO(\upset 0) \leq e(0) = 0$ by (1). Next, observe that by compactness $e\SO(\soI)\vee e\SO(\soII) \leq e(a)$ iff $a\in \soI\cap \soII$. Consequently, if $e\SO(\soIII) \leq e\SO(\soI\cap \soII)$ and $e\SO(\soI)\vee e\SO(\soII) \leq e(a)$, then $e\SO(\soIII) \leq e\SO(\soI\cap \soII) \leq e(a)$ and so $e\SO$ also preserves binary joins.

    Finally, we show that $e\SO$ preserves filtered meets. Let $\p F \sue L\SO$ be filtered, and assume that $e\SO(\soIII) \leq \bigmee\dirdown\, e\SO[\p F]$ and $e\SO(\bigmee\dirdown\, \p F) = e\SO(\dircup \p F) \leq e(a)$. From the latter inequality we have that $a\in \soI$ for some $\soI \in \p F$. As a consequence $e\SO(\soIII) \leq \bigmee\dirdown\, e\SO[\p F] \leq e\SO(\soI) \leq e(a)$.

    (3) Let $v = \bigmee \{ e(a) ~|~ a\in L,\ u \leq e(a) \}$. We will show that $v\leq u$. Let $\soI \in L\SO$ and $b\in L$ be such that $e\SO(\soI) \leq v$ and $u\leq e(b)$. From the definition of $v$ we immediately see that $e\SO(\soI) \leq v \leq e(b)$.

    The proof of the second part is analogous.

    (4) What is left to prove is that $e$ preserves binary joins. Let $e\SO(\soI) \leq e(a\vee b)$ and $e(a)\vee e(b) \leq e(c)$. Since $e$ is injective, it follows that $a\leq c$. Indeed, because $e$ preserves finite meets, $e(a\mee c) = e(a) \mee e(c) = e(a)$ and so $a\mee c = a$. Similarly, $b\leq c$. Moreover, we know from $e\SO(\soI) \leq e(a\vee b)$ that $a\vee b \in \soI$. Consequently, $c\in \soI$ and so $e\SO(\soI) \leq e(c)$.

    (5) It is enough to show injectivity of \emb{}. Assume that $e(a) = e(b)$ for some $a, b\in L$. Let $c \wbelow a$. This means that there is a Scott-open filter ${\soI\in L\SO}$ such that $a\in \soI\sue \upset c$. Because $e\SO(\soI) \leq e(a)$, also $b\in \soI$ by compactness and so $c\leq b$. Consequently $a\leq b$ since $a = \bigvee \{ c ~|~ c\wbelow a \}$. The other direction is symmetrical.
\end{proof}

It might seem surprising how much we can prove from density and compactness alone. This is a common feature of canonical extensions. Another common feature is usually that the extension is injective. We showed that this is achieved whenever the frame is locally compact. Moreover, as we discuss in Section~\ref{s:injectivity}, injectivity is in fact equivalent to a certain topological condition.

\subsubsection{A few more observations}

\begin{lemma}
Let \emb{} be a canonical extension of a frame $L$, then
    \[ \bigvee_{i\in I} e\SO(\soI_i) = \bigmee e[\bigcap_{i\in I} \soI_i],\]
    for every collection $\{\soI_i\}_{i\in I} \sue L\SO$.
\end{lemma}
\begin{proof}
    The ``$\leq$'' inequality is immediate as $e\SO(\soI_i) = \bigmee e[\soI_i] \leq \bigmee e[\bigcap_{i\in I} \soI_i]$, for every $i\in I$. We show the other inequality by density. Let $e\SO(\soII)\leq \bigmee e[\bigcap_{i\in I} \soI_i]$ and $\bigvee_{i\in I} e\SO(\soI_i)\leq e(a)$ for some $\soII\in L\SO$ and $a\in L$. The former assumption implies, by compactness, that $\bigcap_{i\in I} \soI_i \sue \soII$ and the latter implies that $a\in \soI_i$, for all $i\in I$. Therefore, $a\in \bigcap_{i\in I} \soI_i \sue \soII$ and so $\bigmee e[\bigcap_{i\in I} \soI_i] \leq e(a)$.
\end{proof}

One can also reformulate the second axiom of canonical extensions for frames to be more similar to the one for distributive lattices (Section~\ref{s:abstract-dlat}).

\begin{lemma}\label{l:compact-open}
    Let $L$ be a frame and \emb{} be a monotone map satisfying density. Then, $e$ is compact if and only if it satisfies
\begin{axioms}
    \item[(Compact+)] If $\bigmee e\SO[\p F] \leq \bigvee e[D]$ for a filtered $\p F\sue L\SO$ and directed $D\sue L$, then $d\in \soI$ for some $\soI\in \p F$ and $d\in D$.
\end{axioms}
\end{lemma}
\begin{proof}
    (Compact+) implies (Compact) since the latter is just a special case of the former, that is, when we take $\p F=\{\soI\}$ and $D = \{a\}$. The reverse direction is a consequence of Proposition~\ref{p:basic-properties}. We have that
    \[ e\SO(\bigcup\dirup\, \p F) = \bigmee\dirup\, e\SO[\p F]  \leq \dirvee e[D] = e(\dirvee D), \]
    and so, by compactness of $e$, $\dirvee D\in \bigcup \p F$. Hence, $\bigvee D\in \soI$, for some $\soI\in \p F$, and also $d\in \soI$, for some $d\in D$, since $\soI$ is Scott-open.
\end{proof}

This stronger version of compactness also specialises to a well-known consequence of the Hofmann--Mislove Theorem. For a filtered collection of compact saturated sets $\{ K_i \}_i$ and an open set $U$ of a sober space, $\bigcap_i K_i \sue U$ if and only if $K_i\sue U$ already for some $K_i$ (e.g.~\cite{jungsunderhauf96duality} or Theorem~II-1.21 in~\cite{gierzetal2003continuous}).

\medskip
In the next section we show, among others, that canonical extensions of locally compact frames are completely distributive lattices. This fact relies on the axiom of choice. One can prove a weaker property constructively for \emph{stably locally compact frames}, that is, for locally compact frames which satisfy:
\begin{align*}
    (\forall a,b,c)\qquad a \wbelow b \ete{and} a\wbelow c \qtq{implies} a\wbelow b\mee c.
\end{align*}
\begin{proposition}
    Let $L$ be a stably locally compact frame and $\emb$ its canonical extension, then $L\can$ is a frame and coframe.
\end{proposition}
\begin{proof}
    The proof follows the same strategy as the proof of Theorem~3 in~\cite{gehrke2014canonical}. However, we need to repeat the same argument twice because the assumption that $L$ is stably locally compact is used in different parts of the proof when proving that $L^\delta$ is a frame and coframe, respectively.

    Recall that a locally compact frame is stably locally compact if and only if the meet of two Scott-open filters in $\Filt(L)$ (in the reverse inclusion order) is also Scott-open. Observe that $e\SO\colon L\SO \to L\can$ preserves binary meets because, by Proposition~\ref{p:basic-properties}, $e\SO(\soI)\mee e\SO(\soI') = \bigvee \{ e\SO(\soII) \mid \soII\in L\SO,\ \soII\leq \soI,\ \soII \leq \soI'\} = e\SO(\soI\mee \soI')$.

    (1) We show that $L\can$ is a frame. First we prove a restricted version of the frame distributivity law. For $K \sue L\SO$ and $\soI\in L\SO$, we verify that
    \begin{align}
        e\SO(\soI) \mee \bigvee e\SO[K] \leq \bigvee \{ e\SO(\soI) \mee e\SO(\soI') \mid \soI' \in K\}.
        \tag{$\star$1}
        \label{eq:frame-simple}
    \end{align}
    By density it is enough to show that $e\SO(\soII) \leq e(a)$ for every $\soII\in L\SO$ and $a\in L$ such that $e\SO(\soII) \leq e\SO(\soI) \mee \bigvee e\SO[K]$ and $\bigvee \{ e\SO(\soI) \mee e\SO(\soI') \mid \soI' \in K\}\leq e(a)$. Because $L$ is stably locally compact, $e\SO(\soI)\mee e\SO(\soI') \leq e(a)$ iff $b\mee b' \leq a$ for some $b\in \soI$ and $b'\in \soI'$. Since $e\SO(\soI)\mee e\SO(\soI') \leq a$ for every $\soI'\in K$,
    \[ \bigvee e\SO[K] \leq \bigvee \{ e(b') \mid b' \in L \text{ s.t. } \exists \soI'\in K, b\in L\colon\ b\mee b' \leq a,\, b \in \soI,\, b' \in \soI' \}. \]
    By Lemma~\ref{l:compact-open}, this together with $e\SO(\soII) \leq e\SO(\soI) \mee \bigvee e\SO[K] \leq \bigvee e\SO[K]$ implies that there are $b_1, \dots, b_n$ and $b'_1, \dots, b'_n$ in $L$ such that $e\SO(\soI) \leq e(b_i)$, $b_i\mee b'_i \leq a$ (for $i=1,\dots,n$) and $e\SO(\soII) \leq e(b'_1\vee \ldots \vee b'_n)$. Therefore, $e\SO(\soII) \leq e\SO(\soI) \leq e(b)$ for $b= b_1 \mee \dots \mee b_n$ and, because $e(\ARG)$ is a frame homomorphism (by Proposition~\ref{p:basic-properties}), also
    \begin{align*}
        e\SO(\soII) &\leq e(b) \mee e(b'_1\vee \ldots \vee b'_n) = e(b\mee b'_1) \vee \ldots \vee e(b\mee b'_n) \\
        &\leq e(b_1\mee b'_1) \vee \ldots \vee e(b_n\mee b'_n) \leq e(a).
    \end{align*}
    This finishes the proof of \eqref{eq:frame-simple}. Since every element of $L\can$ is a join of elements in $L\SO$, to show that $L\can$ is a frame, it is enough to show that $u \mee \bigvee e\SO[K] \leq \bigvee \{ u \mee e\SO(\soI) \mid \soI \in K\}$ for arbitrary $u\in L\can$ and $K\sue L\SO$. From the previous,
    \begin{align*}
        u \mee \bigvee e\SO[K]
            &= \bigvee \{ e\SO(\soI) \in L\SO \mid e\SO(\soI) \leq u \mee \bigvee e\SO[K] \} \\
            &\leq \bigvee \{ e\SO(\soI) \mee \bigvee e\SO[K] \mid \soI \in L\SO,\ e\SO(\soI) \leq u \} \\
            &\leq \bigvee_{e\SO(\soI)\leq u}\ \bigvee_{\soI' \in K} e\SO(\soI)\mee e\SO(\soI') \\
            &= \bigvee_{\soI' \in K}\ \bigvee_{e\SO(\soI)\leq u}  e\SO(\soI)\mee e\SO(\soI')
            \leq \bigvee_{\soI' \in K} u \mee e\SO(\soI').
    \end{align*}

    (2) Showing that $L\can$ is a coframe follows the same pattern. First, we show that, for $a\in L$ and $A \sue L$,
    \begin{align}
        \bigmee \{ e(a)\vee e(a') \mid a' \in A\} \leq e(a)\vee \bigmee e[A].
        \tag{$\star$2}
        \label{eq:coframe-simple}
    \end{align}
    Let $\soI\in L\SO$ and $c\in L$ be such that $e\SO(\soI) \leq \bigmee \{ e(a)\vee e(a') \mid a' \in A\}$ and $e(a)\vee \bigmee e[A] \leq e(c)$. Because $L$ is locally compact and the embedding $e$ is compact, whenever $e\SO(\soI) \leq e(a)\vee e(a') = e(a \vee a')$, for an $a'\in A$, then $e\SO(\soI) \leq e(b)\vee e(b')$ for some $b, b'\in L$ such that $b\wbelow a$ and $b'\wbelow a'$. Moreover, $b \wbelow a$ iff there is some $\soII \in L\SO$ such that $a\in \soII \sue \upset b$ or, in other words, $e(b) \leq e\SO(\soII) \leq e(a)$ in $L\can$. Similarly, there is a $\soII'\in L\SO$ such that $e(b') \leq e\SO(\soII') \leq e(a')$ and so $e\SO(\soI) \leq e\SO(\soII) \vee e\SO(\soII')$. Consequently, since $e\SO(\soI) \leq e(a)\vee e(a')$ for every $a'\in A$,
    \[ \bigmee \{ e\SO(\soII') \mid \soII' \in L\SO \text{ s.t.}\,\, \exists a'\in A, \soII\in L\SO\colon a\in \soII,\, a' \in \soII',\, \soI \leq \soII\vee \soII' \} \leq \bigmee e[A] \]
    Since $L$ is stably locally compact, non-empty finite meets of the elements in the set on the left-hand side form a directed family in $L\SO$. Therefore, from $\bigmee e[A] \leq e(a)\vee \bigmee e[A] \leq e(c)$ it follows, by Lemma~\ref{l:compact-open}, that there are some $\soII_1, \dots, \soII_n$ and $\soII'_1, \dots, \soII'_n$ in $L\SO$ such that $e\SO(\soII_i) \leq e(a)$, $\soI \leq \soII_i\vee \soII'_i$ (for $i=1,\dots,n$) and $e\SO(\soII') \leq e(c)$ where $\soII' = \soII'_1\mee \dots\mee \soII'_n$. We obtain that
    \begin{align*}
        e\SO(\soI)
        &\leq e\SO(\soII_1 \vee \soII'_1) \mee \dots \mee e\SO(\soII_n \vee \soII'_n) \\
        &\leq (e(a) \vee e\SO(\soII'_1)) \mee \dots \mee (e(a) \vee e\SO(\soII'_n))
        = e(a) \vee e\SO(\soII') \leq e(c),
    \end{align*}
    where the only equality holds because $L\can$ is a frame. This finishes the proof of \eqref{eq:coframe-simple}. The rest is the same as in part (1).
\end{proof}

\section{Existence, uniqueness and concrete description}\label{s:concrete-description}

As we explained at the beginning of Section~\ref{s:abstract-frames}, we expect that the canonical extension \emb{} of a frame $L$ can be obtained by a construction by polarities similarly to distributive lattices and spaces. That is, we expect that
\[ L\can \cong \p G(L\SO, L, Z) \qtq{where} \soI Z a \ete{iff} a\in \soI. \]
(Recall the definition of $\p G(\ARG,\ARG,\ARG)$ from Section~\ref{s:abstract-dlat}.)

This follows from a standard argument, e.g.\ \cite{gehrkejansanapalmigiano2013delta}. Let us denote $\p G(L\SO, L, Z)$ by $C$. The crucial ingredient in the construction of $C$ is the pair of antitone maps $p$ and $q$ whose definitions, in this case, unfold as
\begin{align*}
    p\colon& \Ps(L\SO) \to \Ps(L),\quad M\mapsto \{ a\in L ~|~ \forall \soI\in M.\ a\in \soI\} = \bigcap M, \ete{and}\\
    q\colon& \Ps(L) \to \Ps(L\SO),\quad N\mapsto \{ \soI\in L\SO ~|~ \forall a\in N.\ a\in \soI\} = \{ \soI ~|~ N \sue \soI\}.
\end{align*}
Notice that the image of $p$ is always a filter on $L$ and, similarly, an image of $q$ is an ideal on $L\SO$, when $L\SO$ is taken to be ordered by set inclusion. Therefore, it is safe to restrict the adjoint maps to:
\begin{diagram}
    \Idl(L\SO) \ar[bend left=25]{rr}{p} \ar[draw=none, labels=description]{rr}{\bot} & & \Filt(L) \ar[bend left=25]{ll}{q}
\end{diagram}
Since we order $\Filt(L)$ by reverse inclusion, the antitone Galois adjunction becomes monotone, with $p\colon \p I \mapsto \bigvee \p I$ (evaluated in $(\Filt(L), \supseteq)$) on the left and $q$ on the right.
The maps $f\colon L\SO \to C$ and $g\colon L\to C$ are defined as
\begin{align*}
    f&\colon  \soII\mapsto qp(\{\soII\}) = q(\soII) = \{ \soI \in L\SO ~|~ \soI \supseteq \soII\} \qtq{and}\\
    g&\colon a\mapsto q(\{a\}) = \{ \soI\in L\SO ~|~ a\in \soI\}.
\end{align*}

In Proposition~\ref{p:basic-properties} we showed that having a canonical extension \emb{} of $L$ guarantees existence of a map $e\SO\colon L\SO \to L\can$ such that, for $f=e\SO$ and $g=e$, the conditions (1) and (2) of Fact~\ref{f:polarities} are satisfied (where (2) follows from compactness). This means, by Fact~\ref{f:polarities}.3, that if \emb{} is a canonical extension, it must be that $L\can$ is isomorphic to $C$. The existence of a canonical extension map $L\to L\can$ follows from the following lemma.

\begin{lemma}
    The map $g\colon L \to C$ defined above is the canonical extension of $L$, that is, it is dense and compact (in the sense of Definition~\ref{d:can-ext}).
\end{lemma}
\begin{proof}
    Recall that, since $C$ consists of the fixpoints of the closure operator $q(p(\ARG))$ on the frame $\Idl(L\SO)$, the meets in $C$ are computed as in $\Idl(L\SO)$. Hence, $f(\soI) = \bigmee g[\soI]$, for every $\soI\in L\SO$, since
    \[ \bigmee g[\soI] = \bigcap g[\soI] = \bigcap \{ g(a) ~|~ a\in \soI\} = \{ \soII ~|~ \forall a\in \soI.\ a\in \soII\} = f(\soI).\]

    As a consequence, $g$ is compact, by Fact~\ref{f:polarities}.1, and density follows from Fact~\ref{f:polarities}.2 and the fact that $f(\soI) \leq g(a)$ iff $a\in \soI$.
\end{proof}

To summarise, we have obtained the following.

\begin{theorem}\label{t:unicity-existence}
    Let $L$ be a frame. Its canonical extension \emb{} exists, it is uniquely determined, and $L\can\cong \p G(L\SO, L, Z)$ where $\soI Z a$ iff $a\in \soI$.
\end{theorem}

When a frame $L$ is locally compact, then we know more about the structure of $L\can$. The embedding $L \into \Up(\pt(L))$ is also a canonical extension (Example~\ref{e:can-ext-space}) and therefore the frames $L\can$ and $\Up(\pt(L))$ are isomorphic (by Theorem~\ref{t:unicity-existence}). Because the lattice of saturated subsets is completely distributive, the lattice $L\can$ is too.

\begin{corollary}\label{c:completely-distributive}
    The canonical extension of a locally compact frame is a completely distributive lattice.
\end{corollary}

The concrete representation of $L\can$ as the fixpoints of the nucleus $q(p(\ARG))$ on $\Idl(L\SO)$ can be equivalently described as the fixpoints of $p(q(\ARG))$ on $\Filt(L)$. Observe that a filter $F\in \Filt(L)$ is a fixpoint of $p(q(\ARG))$ iff $F = \bigcap \{ \soI ~|~ F \sue \soI\}$. In other words, fixpoints of $p(q(\ARG))$ are precisely the intersections of Scott-open filters.

\begin{corollary}\label{c:filt-repre}
$L\can$ is isomorphic to the poset of filters obtained as intersections of Scott-open filters on $L$, ordered by reverse inclusion.
\end{corollary}

\section{On injectivity of \emb{}}\label{s:injectivity}

Let us explore when is the canonical extension map \emb{} injective. For $a, b\in L$, by Proposition~\ref{p:basic-properties}, $e(a) = e(b)$ if and only if, for all $\soI\in L\SO$, $e\SO(\soI) \leq e(a)$ iff $e\SO(\soI) \leq e(b)$, and this expression is (by compactness of $e$) equivalent to: $a\in \soI$ iff $b\in \soI$. Since Scott-open filters are in a bijective correspondence with preframe homomorphisms to $\two$, $e$ is injective if and only if the frame $L$ is spatial as a preframe. Furthermore, injectivity of canonical extensions can be also characterised purely in terms of sublocales.
\begin{proposition}\label{p:injective-subloc}
For a frame $L$, the canonical extension \emb{} is injective if and only if
\[  U \sue V \qtq{iff} \forall \text{compact fitted\,\footnotemark sublocale }K\sue L,\quad  K\sue U \text{ implies } K \sue V, \]
for all open sublocales $U$, $V\sue L$.

    Moreover, the Ultrafilter Principle is equivalent to the statement that every frame with injective canonical extension is spatial.
\end{proposition}
\begin{proof}
    From the previous discussion we know that $e$ is injective iff distinct elements of $L$ can be separated by Scott-open filters. By Johnstone's pointfree version of Hofmann--Mislove Theorem~(Lemma~3.4 in~\cite{johnstone1985vietoris}, see also~\cite{escardo2003joins}), Scott-open filters are in a bijective correspondence with compact fitted sublocales, via the correspondence
    \[ \soI \mapsto K_\soI = \bigcap \{ U_a ~|~ a\in \soI\}, \qtq{and} K \mapsto \{ a ~|~ K \sue U_a\}. \]
    From this it follows that the required separation property is equivalent to $e$ being injective because, for every $a\in L$ and $\soI\in L\SO$, we have that $a\in \soI$ iff $K_\soI \sue U_a$.

    The second part of the statement follows from Ern\'e's Theorem 5.2 in~\cite{erne2007choiceless} that the Ultrafilter Principle is equivalent to the statement that every frame spatial as a preframe is spatial as a frame.
\end{proof}
\footnotetext{Fitted means that it is an intersection of open sublocales. Being fitted is supposed to be the sublocale analogue of being a saturated subset.}

This proposition justifies that we can call frames with injective canonical extensions as frames which \emph{have enough compact fitted sublocales}.\footnote{Such a long name almost begs to be shortened to \emph{has enough cofis} (read as ``coffees'').}

\begin{remark}
    This notion is closely related to Mart\'in Escard\'o's definition of compactly generated locales~\cite{escardo2006compactly}. He proves that a Hausdorff locale is isomorphic to a colimit of its compact fitted sublocales if and only if its open sublocales are determined by the compact sublocales they contain.
\end{remark}

\begin{remark}
    Our construction of canonical extensions, as described in Corollary~\ref{c:filt-repre}, is exactly the same as the one used by Marcel Ern\'e~\cite{erne2007choiceless} to show that a preframe $P$ is spatial precisely whenever there is a complete lattice which is, in Ern\'e's terminology, its $\delta$-sober envelope. This means, in our terminology, that there is an embedding $e\colon P \into C$ into a complete lattice $C$ which is dense and compact.

    In fact, from the beginning of our investigations we were hoping to find connections between \cite{erne2007choiceless} and the theory of canonical extensions, as was foreseen by Mai Gehrke (private communication).
\end{remark}

\begin{remark}
    Assuming the Ultrafilter Principle, we obtain that \emb{} is injective if and only if $L$ is spatial. The left-to-right direction follows from Proposition~\ref{p:injective-subloc} and the reverse direction is a consequence of Example~\ref{e:can-ext-space} and Theorem~\ref{t:unicity-existence}.
\end{remark}

\section{Extensions of maps}
\label{s:extensions}

In the theory of canonical extensions, when extending maps between lattices to maps between their canonical extensions, two constructions are usually considered. The so-called $\sigma$-extension and $\pi$-extension are also used to lift additional operations on the lattice to its canonical extension and, depending on which extension is taken, certain types of equations are preserved. Moreover, we will be mostly concerned with extending so-called \emph{perfect} maps and homomorphisms.

Before we define the two extensions let us introduce a new convention. Writing the applications of the maps \emb{} and $e\SO\colon L\SO\into L\can$ everywhere gets tiring at times. For this reason, by abuse of notation, we will write elements of $L\can$ in the image of $e$ without explicitly writing the application of $e(\ARG)$ and, similarly, for the elements of $L\SO$ and applications of $e\SO(\ARG)$. For example, for a monotone map $f\colon L\to M$, $a\in L$, and $\soI\in L\SO$, the expressions
\[ \soI \leq a \qtq{and} \bigmee f[\soI] \leq f(a) \]
are interpreted as
\[ e\SO(\soI) \leq e(a) \qtq{and} \bigmee e[f[\soI]] \leq e(f(a)), \qtq{respectively.} \]
Furthermore, if $f$ is a preframe homomorphism, then taking the preimage of a Scott-open filter gives again a Scott-open filter. In such case, we denote the corresponding monotone map by
\[ f\iSO\colon M\SO \to L\SO,\quad \soII \mapsto f\inv[\soII]. \]
Then, with the convention from above, $f\iSO(\soII) \leq u$ means $e\SO(f\iSO(\soII)) \leq u$.

With this in mind, let us define the two extensions $f^\sigma, f^\pi\colon L\can\to M\can$ for a monotone map $f\colon L\to M$ between two frames $L$ and $M$ as
\begin{align*}
    f^\sigma\colon u &\mapsto \bigvee \{ \bigmee f[\soI] ~|~ \soI\in L\SO,\ \soI \leq u \},\\
    f^\pi\colon    u &\mapsto \bigmee \{ f(a) ~|~ a\in L,\ u \leq a \}.
\end{align*}
These definitions are adapted from the usual definitions of extensions of monotone maps for canonical extensions of distributive lattices, see e.g.\ \cite{gehrke2014canonical}.

In the following we show that calling those maps ``extensions'' is justified by the fact that, under mild assumptions, they both indeed extend the original map.

\begin{lemma}\label{l:exten-squares}
    Let $h\colon L\to M$ be a preframe homomorphism and $f\colon L\to M$ a monotone map between two frames. The diagrams
    \[
    \cdmatrix{
        L \ar{r}{h}\ar[swap]{d}{e_L} \& M\ar{d}{e_M} \\
        L\can \ar{r}{h^\sigma} \& M\can
    }
    \qqtqq{and}
    \cdmatrix{
        L \ar{r}{f}\ar[swap]{d}{e_L} \& M\ar{d}{e_M} \\
        L\can \ar{r}{f^\pi} \& M\can
    }
\]
    commute.
\end{lemma}
\begin{proof}
    Let $a\in L$. The inequality $\soI \leq a$ implies that $a\in \soI$ and $\bigmee h[\soI] \leq h(a)$. Hence, $h^\sigma(a)\leq h(a)$. We prove the reverse direction by density of $e_M$, let $\soII\leq h(a)$ and $h^\sigma(a) \leq b$ for some $\soII\in M\SO$ and $b\in M$. From the first inequality, by compactness, $h(a) \in \soII$ and so $a\in h\inv[\soII] = h\iSO(\soII)$. Since $h\iSO(\soII)\leq a$, we have that $\bigmee h[h\iSO(\soII)]\leq h^\sigma(a) \leq b$. Finally, observe that $\soII \leq \bigmee h[h\iSO(\soII)]$ in $M\can$, because $h[h\iSO(\soII)]\sue \soII$, and so $\soII\leq b$.

    For the second square, $f^\pi(a) = \bigmee \{ f(a') ~|~ a'\in L,\ a\leq a'\} = f(a)$.
\end{proof}

In the following we assume that the order between maps is always the pointwise order. It is immediate to see that the operators $(\ARG)^\sigma$ and $(\ARG)^\pi$ are monotone with respect to this order.

\begin{lemma}\label{l:basic-extension}
    For monotone maps $f\colon L\to M$ and $g\colon M\to N$  between frames,
    \begin{enumerate}
        \item $f^\sigma \leq f^\pi$,
        \item $f^\sigma = f^\pi$ whenever $f$ is a preframe homomorphism,
        \item $g^\pi f^\pi \leq (gf)^\pi$, and
        \item $(\id_L)^\sigma = (\id_L)^\pi = \id_{L\can}$.
    \end{enumerate}
\end{lemma}
\begin{proof}
    (1) Let $u\in L\can$ and let $\soI\leq u \leq a$ for some $\soI\in L\SO$ and $a\in L$. Then $a\in \soI$ and so $\bigmee f[\soI] \leq f(a)$.

    (2) For the reverse direction, let $\soII \leq f^\pi(u)$ and $f^\sigma(u) \leq b$ for some $\soII\in M\SO$ and $b\in M$. Let us denote the filter $\{ a \in L ~|~ u\leq a\}$ by $F_u$. By density of $e_M$, $\soII \leq f^\pi(u) = \bigmee f[F_u]$ is equivalent to $f[F_u] \sue \soII$ and hence also to $F_u \sue f\inv[\soII]$. By our assumption, $f$ is a preframe homomorphism and so $f\iSO(\soII) = f\inv[\soII]$ is a Scott-open filter. Consequently, $f\iSO(\soII) \leq \bigmee F_u = u$ and, therefore, $\soII \leq \bigmee f[f\iSO(\soII)] = f^\sigma(f\iSO(\soII)) \leq f^\sigma(u) \leq b$.

    (3) Let $\soIII \leq g^\pi(f^\pi(u))$ and $(gf)^\pi(u)\leq c$ for some $\soIII\in N\SO$ and $c\in N$. Because $M\can$ is join-generated by elements of $M\SO$, the statement $f^\pi(u) = \bigmee f[F_u] \leq b$ is equivalent to
    \[ (\forall \soII\in M\SO)\quad f[F_u] \sue \soII \qtq{implies} b\in \soII \]
    (this is because, by compactness, $\soII \leq \bigmee f[F_u]$ iff $f[F_u] \sue \soII$). In other words, $f^\pi(u) \leq b$ iff $b\in B$ where
    \[  B = \bigcap \{ \soII\in M\SO ~|~  f[F_u] \sue \soII \}. \]
    Therefore, $\soIII \leq g^\pi(f^\pi(u)) = \bigmee \{ g(b) ~|~ b\in M \text{ s.t.\ } f^\pi(u) \leq b\} = \bigmee g[B]$.
    On the other hand $(gf)^\pi(u)= \bigmee g[f[F_u]]$. Therefore, from $f[F_u] \sue B$ we obtain the desired $\soIII \leq \bigmee g[B] \leq \bigmee g[f[F_u]] \leq b$.

    (4) By (2) it is enough to show the identity for $\id^\pi$, which follows from the definition of $\pi$-extensions as $\id^\pi(u) = \bigmee F_u = u$.
\end{proof}

In the rest of the text, in cases when $f^\sigma$ is equal to $f^\pi$, we denote the extension map by $f\can$.

\subsection{Perfect maps}
In order to proceed with our theory, we have to assume further properties about our maps. We say that a monotone map $f\colon L\to M$ between frames is \emph{perfect} if, for every $\soI\in L\SO$, the smallest filter containing $f[\soI]$ is Scott-open. In other words, $f$ is perfect if the map
   \[ f\SO\colon \soI\mapsto \upset \{ \bigmee M ~|~ M \sue f[\soI] \text{ and $M$ is finite} \} \]
is a well defined map $L\SO\to M\SO$.

Note that, in the locally compact case, this definition agrees with the other definitions for frame homomorphism known under the same name.
\begin{lemma}
    Let $h\colon L\to M$ be a frame homomorphism with $L$ locally compact. The following three statements are equivalent:
    \begin{enumerate}
        \item $c \wbelow a$ implies $h(c) \wbelow h(a)$, for every $a,c \in L$.
        \item $h$ is perfect.
        \item The right adjoint of $h$ preserves directed joins.
    \end{enumerate}
\end{lemma}
\begin{proof}
    The strategy of the proof is the same as in the proof of Theorem~IV-1.4 in~\cite{gierzetal2003continuous}. It needs to be slightly adapted for Scott-open \emph{filters}.
\end{proof}
%
%

Observe that if a monotone map $f\colon L\to M$ is perfect then its $\sigma$-extension can be written in a more compact way:
\begin{align*}
    f^\sigma\colon u &\mapsto \bigvee \{ f\SO(\soI) ~|~ \soI\in L\SO,\ \soI \leq u \}.
\end{align*}

\begin{lemma}\label{l:perfect-extension}
    Let $f\colon L\to M$ and $g\colon M\to N$ be monotone maps between frames. If $f$ is perfect and $g$ preserves finite meets then

    \begin{enumerate}
        \item the diagrams
        \[
          \cdmatrix{
              L\SO \ar{r}{f\SO}\ar[swap,hook]{d}{e\SO_L} \& M\SO\ar[hook]{d}{e\SO_M} \\
              L\can \ar{r}{f^\sigma} \& M\can
          }
          \qqtqq{and}
          \cdmatrix{
              L\SO \ar{r}{f\SO}\ar[swap,hook]{d}{e\SO_L} \& M\SO\ar[hook]{d}{e\SO_M} \\
              L\can \ar{r}{f^\pi} \& M\can
          }
        \]
        commute, and
        \item $(gf)^\sigma \leq g^\sigma f^\sigma$.
    \end{enumerate}
\end{lemma}
\begin{proof}
    (1) For a $\soI\in L\SO$, clearly $f^\sigma(\soI) = \bigvee \{ f\SO(\soI') ~|~ \soI'\in L\SO,\ \soI' \leq \soI\} = f\SO(\soI)$ and $f^\pi(\soI) = \bigmee \{ f(a) ~|~ a\in L,\ \soI \leq a \} = \bigmee f[\soI]$ which is equal to $f\SO(\soI)$ when interpreted by the map $e\SO\colon L\SO \into L\can$.

    (2) Let $u\in L\can$ and let $\soI\in L\SO$ be such that $\soI \leq u$. By monotonicity of $f^\sigma$ and by (1) we have that $f\SO(\soI) = f^\sigma(\soI) \leq f^\sigma(u)$. Therefore,
    \[ \bigmee g[f\SO(\soI)] \leq \bigvee \{ \bigmee g[\soII] ~|~ \soII \leq f^\sigma(u) \} = g^\sigma(f^\sigma(u)). \]
    Moreover, because $g$ preserves finite meets, the leftmost expression simplifies as follows
    \begin{align*}
        \bigmee g[f\SO(\soI)]
            &= \bigmee \{ g(\bigmee M) ~|~ M \sue f[\soI] \text{ and $M$ is finite}\} \\
            &= \bigmee \{ \bigmee g[M] ~|~ M \sue f[\soI] \text{ and $M$ is finite}\}
            = \bigmee gf[\soI].
    \end{align*}
    Since $(g f)^\sigma(u) = \bigvee \{ \bigmee g f[\soI] ~|~ \soI \leq u\}$ we get that $(g f)^\sigma(u) \leq g^\sigma(f^\sigma(u))$.
\end{proof}

\begin{proposition}\label{p:perfect-is-adjoint}
    Let $f\colon L\rightleftarrows M \colon g$ be monotone maps such that $f$ is the left adjoint of $g$. If $f$ is perfect then $f^\sigma$ is the left adjoint of $g^\pi$.
\end{proposition}
\begin{proof}
    Recall that $f$ and $g$ being adjoint means that $fg\leq \id$ and $\id \leq gf$. Therefore, by Lemma~\ref{l:basic-extension}, $f^\sigma g^\pi \leq f^\pi g^\pi \leq (fg)^\pi \leq \id^\pi = \id$. Because right adjoins preserve arbitrary meets, we can use Lemma~\ref{l:perfect-extension} to show the other inequality: $\id = \id^\sigma \leq (g f)^\sigma \leq g^\sigma f^\sigma \leq g^\pi f^\sigma$.
\end{proof}

\begin{theorem}\label{t:lift-perfect-homo}
    If $h\colon L\to M$ is a perfect frame homomorphism between two frames, then $h\can\colon L\can \to M\can$ is a complete lattice homomorphism, that is, it preserves arbitrary joins and meets.
\end{theorem}
\begin{proof}
    From Proposition~\ref{p:perfect-is-adjoint} we know that $h\can$ is a left adjoint. Therefore, $h\can$ preserves arbitrary joins. What is left to show is that $h\can$ also preserves arbitrary meets.

    Let $\{u_i\}_{i\in I} \sue L\can$, and let $\soII \leq \bigmee_i h\can(u_i)$ and $h\can(\bigmee_i u_i) \leq b$ for some $\soII \in M\SO$ and $b\in M$. Denote the filter $\{ a\in L ~|~ u_i\leq a\}$ by $F_i$. The inequality $\soII \leq \bigmee_i h\can(u_i) \leq h\can(u_i) = \bigmee h[F_i]$ implies that, by compactness of $e_M$, $h[F_i] \sue \soII$ and so $F_i \sue h\inv[\soII] = h\iSO(\soII)$. Because, for all $i\in I$, $h\iSO(\soII) \leq \bigmee F_i = u_i$, we obtain that $h\iSO(\soII) \leq \bigmee_i u_i$. By compactness of $e_L$, $G \sue h\iSO(\soII)$, where $G = \{ a\in L ~|~ \bigmee_i u_i \leq a\}$, and so $h[G] \sue \soII$. Consequently, $\soII \leq \bigmee h[G] = h\can(\bigmee_i u_i)$.
\end{proof}

From Lemmas~\ref{l:basic-extension} and~\ref{l:perfect-extension}, we obtain that $(\ARG)\can$ preserves the identity morphisms and that $(h g)\can \leq h\can g\can \leq (h g)\can$, whenever $h$ and $g$ are perfect frame homomorphisms. Combining all this with Corollary~\ref{c:completely-distributive} and Theorem~\ref{t:lift-perfect-homo} gives the following.

\begin{corollary}\label{c:functor}
    The mapping $L\mapsto L\can$ and $h\mapsto h\can$ constitutes a functor from the category of locally compact frames and perfect homomorphisms to the category of completely distributive lattices and complete lattice homomorphisms.
\end{corollary}

\begin{remark}
    In the theory of canonical extensions, $\sigma$-- and $\pi$-extensions are often used in order to lift operations on a given lattice to its canonical extension. It seems that most of the arguments would go through for extensions of perfect operations on a frame. However, we leave checking whether this is indeed the case for future investigations.
\end{remark}

\subsection{Perfect sublocales}
The map extensions can be used to lift certain sublocales of $L$ to sublocales of $L\can$. We say that a sublocale $S\sue L$ is \emph{perfect} if the associated quotient homomorphism $L\onto S$ is perfect.

\begin{lemma}
    For a perfect sublocale $S\sue L$ of a locally compact frame $L$, $S\can$ is isomorphic to a sublocale of $L\can$.
\end{lemma}
\begin{proof}
    We show that the extension $h\can$ of the perfect onto frame homomorphism $h\colon L\onto S$ is also onto. Since $S\SO$ join-generates $S\can$ and since $h\can$ preserves joins, it is enough to show that $h\can$ maps onto $S\SO$. Let $\soII \in S\SO$, then $h\can(h\iSO(\soII)) = \bigmee \{ h(a) ~|~ h\iSO(\soII) \leq a\} = \bigmee h[h\iSO(\soII)]$ which is equal to $\soII$ because $h$ is onto and so $h[h\iSO(\soII)] = \soII$.
\end{proof}

Note that perfect sublocales have been previously studied, e.g., in~\cite{escardo1999compact,escardo2001regular} it is shown that the pointfree analogue of the patch topology for a stably compact space is precisely the frame of all perfect nuclei.

\section{Monotone maps and spectra}

One of the reasons for studying canonical extensions of frames is that it allows us to represent monotone functions with respect to the specialisation order directly in the category of frames.

Let $(X, \tau)$ and $(Y, \rho)$ be locally compact sober spaces. It is known that monotone maps $(X, \leq_\tau) \to (Y, \leq_\rho)$ are in a bijective correspondence with continuous maps $(X, \Up(X)) \to (Y, \rho)$. Moreover, since $\tau\can \cong \Up(X)$ (by Example~\ref{e:can-ext-space} and Theorem~\ref{t:unicity-existence}), we also have that
\[ \Top((X, \Up(X)),\, (Y, \rho)) \cong \Frm(\rho,\, \Up(X)) \cong \Frm(\rho, \tau\can), \]
where $\Top(X, Y)$ is the set of continuous maps $X \to Y$ and $\Frm(L, M)$ is the set of frame homomorphisms $L \to M$. In other words, we obtain:

\begin{proposition}
    For locally compact frames $L$ and $M$, monotone maps $\pt(L) \to \pt(M)$, with respect to the specialisation orders, are in a bijective correspondence with the frame homomorphisms $M\to L\can$.
\end{proposition}

Next, for a locally compact frame $L$ and its topological dual $(X, \tau)$, we know that the embeddings \embI{} and $L \into \Up(\pt(L)) \cong \Up(X)$ are isomorphic but this does not mean that the spaces $(X, \Up(X))$ and $\pt(L\can)$ are homeomorphic. The embedding $L \into \Up(X)$ arises from the continuous map $(X, \Up(X)) \to (X,\tau)$ but, as we will see, the onto map $\pt(e)\colon \pt(L\can)\onto \pt(L)$ can in general be a different map.

In the following we examine the relationship between the spaces $(X, \Up(X))$ and $\pt(L\can)$.
First, let us look at the interplay between the spectra of the original frame and its canonical extension.

\begin{observation}\label{o:lift-points}
    Let $L$ be any frame and let $\two$ be the two element frame.
    \begin{enumerate}
        \item \two{} is isomorphic to its canonical extension $\two\can$.
        \item All frame homomorphisms $L\to \two$ are perfect.
    \end{enumerate}
\end{observation}
\begin{proof}
    (1) Since $\two\SO \cong \two$, it is easy to check that the identity mapping $\two \to \two$ is both dense and compact. (2) The image of a Scott-open filter under a homomorphism $L\to \two$ is automatically a Scott-open filter on \two{}.
\end{proof}

Let $L$ be a locally compact frame. The frame embedding \embI{} gives rise to a continuous map
\begin{align*}
    \pt(e)\colon & \pt(L\can) \to \pt(L),\quad r\colon L\can\to \two \ee\longmapsto r\circ e\colon L\to \two.
\intertext{An immediate consequence of Observation~\ref{o:lift-points} is that there is also a map in the opposite direction:}
    & \pt(L) \to \pt(L\can),\quad p\colon L\to \two \ee\longmapsto p\can\colon L\can \to \two\can \cong \two.
\end{align*}
It is also immediate to check that $p\can(u)$, for a $u\in L\can$, is computed as $\bigmee p[F_u]$ where $F_u = \{ a\in L ~|~ u \leq a \}$ (we follow the convention from Section~\ref{s:extensions}).

As a result, we obtain that $\pt(e)$ is onto. Indeed, let $p\colon L\to \two$ be a frame homomorphism. For an $a\in L$,
\[ (\pt(e)(p\can))(a) = p\can(e(a)) = \bigmee p[F_a] = \bigmee p[\upset a] = p(a). \]
This is no surprise as $e$ is injective. However, because $p\can$ preserves all meets (Corollary~\ref{c:functor}), we see that $\pt(e)$ stays onto even if we restrict it to the subspace $\ptC(L\can)$ of frame homomorphism $L\can \to \two$ which preserve all meets. This makes sense since $L\can$ lives in the category of completely distributive lattices and so its spectrum should be computed there.

To show that $\pt(e)$ is injective, when restricted to $\ptC(L\can)$, let $r\colon L\can \to \two$ be a complete lattice homomorphism and let $u\in L\can$. We have that
\[ (r\circ e)\can (u) = \bigmee_{a\in F_u} (re)\can(a) = \bigmee_{a\in F_u} r(a) = r(u) \]
where the first equality holds because $(r\circ e)\can$ preserves meets and $u$ is equal to the meet $\bigmee F_u$ in $L\can$, the second equality follows from Lemma~\ref{l:exten-squares} and the last equality holds because $r$ preserves meets.

We have showed that there is a bijection between the points of $\pt(L)$ and $\ptC(L\can)$. This bijection on points extends to a homeomorphism, once we equip $\pt(L)$ with the topology of upsets.

\begin{proposition}
    Let $L$ be a locally compact frame and $X$ be its dual space. Then, $\ptC(L\can)$ (seen as a subspace of $\pt(L\can)$) is homeomorphic to $(X, \Up(X))$.
\end{proposition}
\begin{proof}
    We know that $L\can \cong \Up(X)$. What is left to show is that the points of $\ptC(L\can)$ separate opens. Let $u\not\leq v$ in $L\can$. By density of canonical extensions, there are $\soI\in L\SO$ and $a\in L$ such that $\soI \leq u$, $v\leq a$ and $a\notin \soI$. By a variant of prime filter extension theorem (see, e.g.\ Proposition~VII.6.2 in~\cite{picadopultr2012book}), there exists a frame homomorphism $p\colon L\to \two$ such that $p[\soI] = \{1\}$ and $p(a) = 0$. Consequently, $p\can(u) \geq p\can(\soI) = \bigmee p[\soI] = 1$ and $p\can(u) \leq p\can(a) = 0$.
\end{proof}

\section{Comparison with earlier approaches}

\subsection{Canonical extensions of distributive lattices revisited}\label{s:dlat-revisited}

In this section we show how our construction generalises some of the known constructions of canonical extensions. First, let $A$ be a distributive lattice. Since $\Idl(A)$ is coherent and therefore also locally compact, we have the embedding
\[ i\colon A \xhookrightarrow{\ee{a \ \mapsto\  \downset a}} \Idl(A) \xhookrightarrow{\qq e} \Idl(A)\can. \]
Moreover,  $\Filt(A)$ is known to be isomorphic to $\Idl(A)\SO$ via the map that sends $F$ to the Scott-open filter $\soI_F = \{ I ~|~ I\cap F \not= \emptyset \}$ (e.g., as a special case of Lemma 15 in~\cite{jungsunderhauf96duality}). Therefore, for a filter $F$,
\[ \bigmee i[F] = \bigmee \{ e(\downset a) ~|~ a\in F \} = \bigmee \{ e(I) ~|~ I \cap F \not= \emptyset \} = \bigmee e[\soI_F] \]
where the middle equality follows from the fact that, whenever $I\cap F \not= \emptyset$, then $e(\downset a) \leq e(I)$ for any $a\in I\cap F$. As a consequence, from density and compactness of $e$ we obtain density and compactness of the distributive lattice embedding $i\colon A\into \Idl(A)\can$ (as defined in Section~\ref{s:abstract-dlat}). In other words, the embedding $i\colon A\into \Idl(A)\can$ is a canonical extensions of the distributive lattice $A$ and those are uniquely determined.

\begin{proposition}\label{p:dlat-frmcan}
    Let $A$ be a distributive lattice. Its canonical extension $e\colon A\into A\can$ is isomorphic to $\Idl(A)\can$.
\end{proposition}

Note also that homomorphisms between distributive lattices lift to perfect frame homomorphisms between their frames of ideals and so Theorem~\ref{t:lift-perfect-homo} also applies for them.

\subsection{Canonical extensions of join-strong proximity lattices}

Sam J.\ van Gool defined in~\cite{gool2012duality} canonical extensions for \emph{join-strong proximity lattices}, that is, for lattices $A$ equipped with a relation $R\sue A\ttimes A$ such that, for every $a, a', b, b'\in A$,\footnote{Sam J.\ van Gool shows that the categories of distributive join-strong proximity lattices and stably compact spaces are dually equivalent, which strengthens the result of~\cite{jungsunderhauf96duality} where also the dual of the axiom (5) is required. The condition $R \sue (\leq)$ can be dropped if we slightly modify the definition of round ideals and filters (see~\cite{jungsunderhauf96duality}).}
\begin{enumerate}
    \item $R\circ R = R$ and $R \sue (\leq)$,
    \item $0 R a$ and $a R 1$,
    \item $a R b$ and $a' R b$ iff $(a\vee a') R b$,
    \item $a R b$ and $a R b'$ iff $a R (b\mee b')$,
    \item $a R (b\vee c)$ implies that $\exists b', c'$ such that $c' R c$, $b' R b$ and $a R (b'\vee c')$.
\end{enumerate}
In case when the lattice part of a join-strong proximity lattice $(A, R)$ is distributive, we can use a similar argument as above, to show that the lattice embedding\footnotemark{}
\[ i\colon A \xhookrightarrow{\ee{a \ \mapsto\  I_a}} \RIdl(A) \xhookrightarrow{\qq e} \RIdl(A)\can\]
satisfies Sam van Gool's conditions of canonical extensions for join-strong proximity lattices. Here $\RIdl(A)$ denotes the frame of all \emph{round} ideals, that is, ideals $I$ such that $a\in I$ implies $a'\in I$ for some $aRa'$. The ideal $I_a$ is taken to be $\{ b ~|~ b R a \}$. The argument also relies on the fact that the preframe of Scott-open filters of $\RIdl(A)$ is isomorphic to the frame of round filters (Lemma 15 in~\cite{jungsunderhauf96duality}).
\footnotetext{We know that $i\colon A\into \RIdl(A)\can$ is a lattice homomorphism because $a\mapsto I_a$ is a lattice homomorphism by Proposition 17 in~\cite{jungsunderhauf96duality}, and also that $e$ is an injective frame homomorphism because $\RIdl(A)$ is locally compact by Theorem 11 in~\cite{jungsunderhauf96duality}.}

\begin{proposition}
    Let $(A, R)$ be a distributive join-strong proximity lattice. Its canonical extension $e\colon A\into (A,R)\can$ is isomorphic to $\RIdl(A)\can$.
\end{proposition}

Sam van Gool also shows that his canonical extension lifts the natural morphisms of join-strong proximity lattices (the so-called \emph{j-morphisms}) to maps between the corresponding canonical extensions. Such lifted maps preserve all meets and, when assuming the axiom of choice, it can be shown that they also preserve all joins~\cite{gool2012duality}. Since j-morphisms between proximity lattices correspond exactly to frame homomorphisms between the corresponding stably compact frames, our Theorem~\ref{t:lift-perfect-homo} improves van Gool's result in sense that we do not need to use any choice principle for perfect morphisms.

\subsection{The Boolean case}

It turns out that canonical extensions of Boolean algebras ${B \into B\can}$, which we showed how to obtain by frame-theoretic means in Section~\ref{s:dlat-revisited}, also have a different frame-theoretic description. First, recall that a frame is \emph{subfit} if it satisfies the following condition
\[ (\forall a, b)\qquad a\not\leq b \implies \exists c\ \ a\vee c = 1 \ete{and} b\vee c \not= 1 \]
Let us refer to a recent result by Richard Ball and Ale\v s Pultr~\cite{ballpultr2018maximal} for subfit frames.
\begin{fact}\label{f:Sc-ess-ext}
    Let $L$ be a subfit frame and $M$ a complete Boolean algebra. If $L$ embeds into $M$ in such a way that, whenever $x < y$ in $M$, then there are $a < b$ in $L$ such that
    \[ x \mee b \leq a \qtq{and} y \vee a \geq b, \]
    then $M$ is isomorphic to $\Sc(L)$, which is the frame of sublocales of $L$ which are joins of closed sublocales.
\end{fact}

Moreover, Ball and Pultr~\cite{ballpultr2018maximal} also show that $\Sc(\Oo(X))\cong \Ps(X)$ for every T$_1$-space $X$. On the other hand, we know that $B\can$ is a complete Boolean algebra which is, assuming the Axiom of Choice, isomorphic to $\Ps(\spec(B))$, where $\spec(B)$ is the Stone dual to $B$. This means that the canonical extension of $B$ (seen as a lattice), the canonical extension of the frame $\Idl(B)$, and $\Sc(\Idl(B))$ are all isomorphic. In fact, we can prove the same statement without the detour to spaces and without the axiom of choice.

\begin{theorem}
    Let $B$ be a Boolean algebra, then $B\can$ (resp.\ $\Idl(B)\can$) is isomorphic to $\Sc(\Idl(B))$.
\end{theorem}
\begin{proof}
    We show that the embedding $e\colon \Idl(B) \into \Idl(B)\can$ satisfies the conditions of Fact~\ref{f:Sc-ess-ext}. Observe that $\Idl(B)\can$ is a complete Boolean algebra since $B\can$ is (Theorem~6.3 in~\cite{gehrkeharding2001bounded}) and since $B\can \cong \Idl(B)\can$ (Proposition~\ref{p:dlat-frmcan}).

    Next, let $x < y$ in $\Idl(B)\can$. By density of $e$, there is an ideal $I\in \Idl(B)$ such that
    \begin{align}
        x \leq e(I) \qtq{and} y\not\leq e(I).
        \label{e:bcan-1}
    \end{align}
    Next, by density of the lattice canonical extension $i\colon B\into \Idl(B) \into \Idl(B)\can$, there is a filter $F\in \Idl(B)$ such that
    \begin{align}
        \bigmee i[F] \leq y \qtq{and} \bigmee i[F]\not\leq e(I) = \bigvee i[I].
        \label{e:bcan-2}
    \end{align}
    Hence, compactness of $i$ gives that $F\cap I = \emptyset$ or, in other words, that $f\not\leq i$ for every $f\in F$ and $i\in I$. Since $f\not\leq i$ iff $i \vee \neg f \not= 1$, we see that the join of ideals $I \vee \neg F$, where $\neg F = \{ \neg f ~|~ f\in F\}$, is not equal to the ideal $\downset 1$. Set $a = e(I\vee \neg F)$ and $b = 1$. Finally we check that the pair $a < b$ satisfies the two required inequalities. By (\ref{e:bcan-1}) we have that
    \[ x \mee b = x \leq e(I) \leq e(I \vee \neg F) = a \]
    and by (\ref{e:bcan-2}) we have that
    \[ y \vee a \geq y \vee e(\neg F) \geq \bigmee i[F] \vee e(\neg F) = 1 = b. \]
    To see why the penultimate equality holds, let $\bigmee i[F] \vee e(\neg F) = \bigmee i[F] \vee \bigvee i[\neg F] \leq i(z)$ for some $z\in B$. We see that $z\in F$ by compactness of $i$. Moreover, because $i(\neg z) \leq \bigvee i[\neg F] \leq i(z)$ and so $\neg z \leq z$, it must be that $z$ is equal to 1 (in $B$).
\end{proof}

Note that the present theorem cannot be generalised to distributive lattices because canonical extensions of those are not complete Boolean algebras in general, as required by Fact~\ref{f:Sc-ess-ext}.

\section{Final remarks}

In this paper we have outlined the basics of the theory of canonical extensions for frames. By the example of how powerful the theory of canonical extensions has shown to be for posets and lattices, it is natural to ask whether this can happen for frames too. For example, already with the present work we can lift operations on frames to their canonical extensions. As in the classical theory of canonical extensions, one should expect that such lifted operations will preserve certain types of equations and also that the original operations correspond to certain relations on the spectrum of the frame (in the spirit of~\cite{gehrkejonsson2004bounded,gehrke2014canonical}).

However, the current results are only a starting point. It is not clear yet how to generalise the theory to arbitrary frames. Our construction works well for locally compact frames or, more generally, for frames with enough compact fitted sublocales (Proposition~\ref{p:injective-subloc}).

Another missing piece of the puzzle concerns $\sigma$-- and $\pi$-extensions of maps. All continuous maps between sober spaces lift to continuous maps between their canonical extensions (the preimage of an upset, by a continuous map, is always an upset). For this reason we would expect that homomorphisms between frames lift to complete lattice homomorphisms between the corresponding canonical extensions. However, we were not able to prove this without assuming perfectness. Maybe a more general construction of canonical extensions needs to be considered first.

Finally, it is well known that a complete Boolean algebra is completely distributive if and only if it is atomic, which is a fact that relies on a non-constructive choice principle. For this reason, one cannot expect to prove that the canonical extension of an arbitrary frame is completely distributive. This is because in the special case when we take the frame $\Idl(B)$, for a Boolean algebra $B$, the canonical extension $\Idl(B)\can$ is known to be a complete Boolean algebra (Theorem~6.3 in~\cite{gehrkeharding2001bounded}) and so proving complete distributivity for $\Idl(B)\can$ requires non-constructive principle. It remains open, however, whether canonical extensions of frames are always frames and coframes. Note that this is always the case for canonical extensions of distributive lattices (Theorem 3 in~\cite{gehrke2014canonical}).

\section*{Acknowledgement}
I would like to thank the participants of the Workshop on Algebra, Logic and Topology in Coimbra 2018 for many stimulating discussions. In particular, I am grateful to Mai Gehrke for suggesting this research topic to me and giving me a lot of invaluable feedback in the process.
I also appreciate comments and suggestions of the anonymous referee which greatly improved the presentation of the paper.

Last but not least, I would like to thank Ale\v s Pultr for introducing me to the ever-fascinating world of pointfree topology.


\end{document}